%% file: uijin08existbict_revised.tex
\documentclass[12pt]{amsart}

\usepackage[paper=a4paper,left=3.0cm,right=3.0cm,top=3.7cm,bottom=3.7cm]{geometry}
\usepackage{float, graphicx, amssymb, dsfont}
\usepackage[all,cmtip,line]{xy}

\newtheorem{thm}{Theorem}[section]
\newtheorem*{thmA}{Theorem \ref{thm:main}}
\newtheorem*{thmB}{Theorem \ref{thm:main2}}
\newtheorem{cor}[thm]{Corollary}
\newtheorem{lem}[thm]{Lemma}
\newtheorem{prop}[thm]{Proposition}

\theoremstyle{definition}
\newtheorem{defn}{Definition}[section]
\newtheorem{exam}[thm]{Example}

\theoremstyle{remark}
\newtheorem{rem}[thm]{Remark}
\newtheorem*{claim}{Claim}

\numberwithin{equation}{section}
\newfloat{Figure}{h}{}[section]

\include{abbrev_1.1}
\newcommand{\HL}{\mathcal{HL}}
\newcommand{\LH}{\mathcal{LH}}
\newcommand{\hsyn}{h_{\mathrm{syn}}}

\newcommand{\modp}{~(\mathrm{mod}~ p)}

\newcommand{\as}{almost specified}
\newcommand{\ass}{almost specified shift}
\newcommand{\aspe}{almost specification property}
\newcommand{\spe}{specification property}

\begin{document}

\title[Open and bi-continuing codes]{On the existence of open and bi-continuing codes}
\author{Uijin Jung}

\address{Department of Mathematical Sciences \\
	Korea Advanced Institute of Science and Technology \\
	Daejeon 305-701 \\
	South Korea}
\email{uijin@kaist.ac.kr}


\begin{abstract}
    Given an \rSofic\ $X$, we show that an \rSFT\ $Y$ of lower entropy is a factor of $X$ \ifff\ it is a factor of $X$ by an open \bict\ code. If these equivalent conditions hold and $Y$ is mixing, then any code from a proper subshift of $X$ to $Y$ can be extended to an open \bict\ code on $X$. These results are still valid when $X$ is assumed to be only an \ass, i.e., a subshift satisfying an irreducible version of the specification property.
\end{abstract}
\maketitle

\section{Introduction and Preliminaries}

Let $X$ and $Y$ be shift spaces. Suppose that there exists a factor code from $X$ to $Y$. If we denote by $h(X)$ the topological entropy of $X$, then $h(X) \geq h(Y)$. Also, whenever $x$ is a periodic point of $X$, there exists a periodic point of $Y$ whose period divides the period of $x$. We denote this condition by $P(X) \searrow P(Y)$ and call it the \emph{periodic condition}. These two conditions (on entropies and periodic points) are necessary for $X$ to factor onto $Y$. To find satisfactory sufficient conditions for $X$ to factor onto $Y$ (by a code with certain properties) is one of the important problems in symbolic dynamics.

In \cite{Boy84}, Boyle proved the following theorem: Given irreducible \SFTs\ $X$ and $Y$ with $h(X) > h(Y)$, $Y$ is a factor of $X$ \ifff\ the periodic condition $P(X) \searrow P(Y)$ holds. One can consider two directions of generalizing this result: One is to impose certain properties on the factor code, and the other is to weaken the finite-type constraints on $X$ and $Y$.

The first type of generalization is due to Boyle and Tuncel. In \cite{BoyT}, they showed that the entropy and the periodic conditions also imply that $Y$ is a factor of $X$ by a \rct\ code. A \emph{\rct\ code} is a code which is surjective on each unstable set (see Definition \ref{def:rct}). It is a natural dual version of a \rc\ code and plays a fundamental role in the class of \ito\ codes \cite{BoyT, Boy05}.

The other direction of generalization is considered by Thomsen \cite{Tho}, who investigated the existence of factor codes between \rSofic s. In particular, he gave a necessary and sufficient condition for an \rSFT\ to be a factor of a given synchronized system of greater synchronized entropy (see \S 5).

In this paper, we focus on two kinds of codes: Open codes and \bict\ codes, i.e., left and right continuing codes. We will generalize the above results in both directions. A code is \emph{open} if it sends an open set to an open set. Finite-to-one open codes between \SFTs\ have been studied in several contexts \cite{Hed,Nas,Jung}. A \fto\ factor code between \rSFTs\ is open \ifff\ it is \bic\ \cite{Nas}. A natural generalization exists: A factor code from a \SFT\ to a shift space is open \ifff\ it is \bict\ (see Theorem \ref{thm:open_iff_bict}). In \S \ref{Section_II:Properties}
, we investigate these two types of codes and give several sufficient conditions for one to imply the other. 

To enlarge the class of domains satisfying the theorem of Boyle and Tuncel, we consider the class of \emph{shifts with the \aspe} (or \emph{\ass s}) which are subshifts satisfying an irreducible version of the specification property (see Definition \ref{def:specification+almost_specification}). Every irreducible sofic shift is \as, and \ass s behave much like shifts with the \spe\ except being mixing. Our main theorem is stated as follows.

\begin{thmA}
    Let $X$ be an \ass\ and $Y$ an \rSFT\ such that $h(X) > h(Y)$. Then the following are equivalent.
    \begin{enumerate}
        \item $Y$ is a factor of $X$ by a \bict\ code with a bi-retract.
        \item $Y$ is a factor of $X$ by an open code with a uniform lifting length.
        \item $Y$ is a factor of $X$.
    \end{enumerate}
    If $X$ is of finite type, then the above conditions are also equivalent to:
    \begin{enumerate}
        \item[(4)] $P(X) \searrow P(Y)$.
    \end{enumerate}
\end{thmA}

As a corollary, it follows that given an \rSFT\ $X$, a shift space $Y$ is a lower entropy open factor of $X$ exactly when $h(X) > h(Y)$, $P(X) \searrow P(Y)$ and $Y$ is an \rSFT\ (see Corollary \ref{thm:exist_open_code}).

It is not surprising that factor problems usually involve extension problems \cite{Ash93, Boy84,BoyT}. Even in the case where $X$ and $Y$ are \rSFTs\ and $P(X) \searrow P(Y)$, not every code can be extended (see Example \ref{ex:extensionfail}). In \S \ref{sec:main2}, we give a necessary condition for a code on a proper subshift of $X$ to be extended to a code on $X$ and show that it is the only obstruction to extend it to a code on $X$. We give an extension theorem as follows.

\begin{thmB}
    Let $X$ be an \ass\ and $Y$ an \rSFT\ such that $h(X) > h(Y)$ and $Y$ is a factor of $X$. Then any code from a proper subshift of $X$ to $Y$, which can be extended to a code on $X$, can be extended to an open \bict\ code of $X$ onto $Y$.
\end{thmB}

In particular, if $Y$ is mixing, then any code from a proper subshift of $X$ to $Y$ can be extended to an open \bict\ code of $X$ onto $Y$ (see Theorem \ref{thm:exist_bict_code_1}).
The theorem says that the openness and the \bict\ property of a code are both global in the sense that these properties cannot be ruled out by any property a code might exhibit on a proper subshift.

\vspace{0.15cm}
We introduce some notations and definitions used. If $X$ is a shift space with the shift map $\sigma$, denote by $\B_n(X)$ the set of all words of length $n$ appearing in the points of $X$ and $\B(X) = \bigcup_{n \geq 0} \B_n(X)$. Also let $\A(X) = \B_1(X)$ and $\B^X_n(a,b) = \{ u_1 \ldots u_n \in \B_n(X) : u_1 = a, u_n = b\}$ for each $a,b \in \A(X)$.
A shift space $X$ is called \textit{irreducible} if for all $u, v \in \B(X)$, there is a word $w$ with $uwv \in \B(X)$. It is called \textit{mixing} if for all $u, v \in \B(X)$, there is an integer $N \in \setN$ such that whenever $n \geq N$, we can find $w \in \B_n(X)$ with $uwv \in \B(X)$. If there is such an $N$ which works for all $u,v \in \B(X)$, then we call $N+1$ a \emph{transition length} for $X$.

Given $u \in \B(X)$, let ${}_l[u]$ denote the open and closed set $\{x \in X : x_{[l,l+|u|-1]}=u\}$, which we call a \emph{cylinder}. If $|u|=2l+1$, then ${}_{-l}[u]$ is called a \emph{central} $2l+1$ \emph{cylinder}. When $u \in \B_1(X)$ and $l=0$, we usually discard the subscript 0.

A \emph{code} $\pXtoY$ is a continuous $\sigma$-commuting map between shift spaces. It is called a \emph{factor code} if it is onto. In this case we say that $Y$ is a \emph{factor} of $X$. Any code can be recoded to be a 1-block code, i.e. a code for which $x_0$ determines $\phi(x)_0$.

A subshift $X$ is called a \textit{shift of finite type} if there is a finite set $\F$ of words such that $X$ consists of all points on $\A(X)$ in which there is no occurrence of words from $\F$. Any shift of finite type is conjugate to an edge shift, i.e., a shift space which consists of all bi-infinite trips in a directed graph $G$. If $A$ is the adjacency matrix for $G$, denote by $\X_A$ the edge shift on $G$. A shift space is called \emph{sofic} if it is a factor of a \SFT. A mixing \Sofic\ always has a transition length.

A word $v \in \B(X)$ is \emph{synchronizing} if whenever $uv$ and $vw$ are in $\B(X)$, we have $uvw \in \B(X)$. Following \cite{BlaH}, an irreducible shift space $X$ is a \emph{synchronized system} if it has a synchronizing word. For a given set of words $W$, let $\X_W$ be the smallest shift space containing all the sequences obtained by concatenating words in $W$. We call $\X_W$ a \emph{coded system}. A synchronized system is coded \cite{BlaH}.

For a shift space $X$ with periodic points, let $\per(X)$ be the greatest common divisor of periods of all periodic points of $X$. Let $X$ be an irreducible edge shift and $p = \per(X)$. Then there exists a unique partition $\{ A_0, A_1, \cdots, A_{p-1} \}$ of $\A(X)$  such that whenever $ab \in \B(X)$ and $a \in A_i$, then $b \in A_{i+1 \modp}$. Also there is $n \in \setN$ with the property that $\B_n^X(a,b) \neq \emptyset$ for any  $a, b \in A_i$ and $i = 0, 1, \cdots, p-1$. This $n$ is called a \emph{weak transition length} for $X$.

The entropy of a shift space $X$ is defined by $h(X) = \lim_{n \to \infty} (1/n) \log |\B_n(X)|$, which equals the topological entropy of $(X,\sigma)$ as a dynamical system.
If $X$ is an \rSFT\ and $p = \per(X)$, then for any $a, b \in \A(X)$ we have the following equalities:
    \[ \begin{split}
            h(X) & = \lim_{n \to \infty} \frac{1}{np} \log |p_{np}(X)| \\
                 & = \lim_{n \to \infty} \frac{1}{np} \log |q_{np}(X)| = \limsup_{n \to \infty} \frac{1}{n}
                 \log |\B_n^X(a,b)|,
       \end{split}
    \]
where $p_k(X)$ (resp. $q_k(X)$) denotes the number of periodic points of $X$ with period $k$ (resp. least period $k$). If $X$ is mixing, then limsup becomes limit.

For more details on symbolic dynamics, see \cite{LM}. For a perspective for open maps between shift spaces, see \cite{Jung}.

\vspace{0.3cm}
\section{Properties of continuing codes}\label{Section_II:Properties}

A code $\pXtoY$ is called \rc\ if whenever $x \in X$, $y \in Y$ and $\phi(x)$ is left asymptotic to $y$, then there exists at most one $\bar x \in X$ such that $\bar x$ is left asymptotic to $x$ and $\phi(\bar x) = y$. Right closing codes form an important class of \fto\ codes, especially between \rSFTs. The following notion, which can be thought as a dual of the above property, first appeared in \cite{BoyT}.

\begin{defn} \label{def:rct}
    A code $\pXtoY$ between shift spaces is called \emph{\rct} if whenever $x \in X$, $y \in Y$ and $\phi(x)$ is left asymptotic to $y$, then there exists at least one $\bar x \in X$ such that $\bar x$ is left asymptotic to $x$ and $\phi(\bar x) = y$. A \emph{\lct} code is defined similarly. If $\phi$ is both left and right continuing, it is called \emph{\bict}.
\end{defn}

A \rct\ code into an irreducible sofic shift must be onto. In general it need not be onto.
The Perron-Frobenius theory implies that a factor code between \rSFTs\ with the same entropy is \rct\ exactly when it is \rc.

\begin{defn}
    An integer $n \in \Zplus$ is called a (\emph{\rct}) \emph{retract} of a \rct\ code $\pXtoY$ if, whenever $x \in X$ and $y \in Y$ with $\phi(x)_{(-\infty,0]} = y_{(-\infty,0]}$, we can find $\bar x \in X$ such that $\phi(\bar x) = y$ and $x_{(-\infty,-n]} = \bar x_{(-\infty,-n]}$.
\end{defn}

Note that having a retract is a conjugacy invariant. If $\phi$ has both left and right continuing retracts $n$, we simply say that $\phi$ has \emph{bi-retract $n$}.

A code $\pXtoY$ between shift spaces is open \ifff\ for each $k \in \setN$, there is $m \in \setN$ such that whenever $x \in X, y \in Y$ and $\phi(x)_{[-m,m]} = y_{[-m,m]}$, we can find $\bar{x} \in X$ with $\bar{x}_{[-k,k]} = x_{[-k,k]}$ and $\phi(x) = y$.
We say that $\phi$ has a \emph{uniform lifting length} if for each $k \in \setN$, there exists $m$ satisfying the above property such that $\sup_k |m-k| < \infty$. In this case, we can assume that $m-k$ is constant and nonnegative. Note that having a uniform lifting length is also a conjugacy invariant.

\begin{lem} \label{lem:open_ull_then_bict_and_converse}
    Let $\pXtoY$ be a code between shift spaces. If $\phi$ is open with a uniform lifting length, then it is \bict\ with a bi-retract.
    When $Y$ is of finite type, then the converse holds.
\end{lem}
\begin{proof}
    First, suppose that $\phi$ is open with a uniform lifting length. We can assume $\phi$ is a 1-block code. Choose $l \geq 0$ so that for each central $2k+1$ cylinder in $X$, its image consists of central $2k+2l+1$ cylinders. Let $x \in X$ and $y \in Y$ satisfy $\phi(x)_{(-\infty,0]} = y_{(-\infty,0]}$. By considering $\phi({}_{-2k-l}[x_{-2k-l} \cdots x_{-l}])$ for $k \in \setN$, we obtain points $z^{(k)}$ such that $z^{(k)}_{[-2k-l,-l]} = x_{[-2k-l,-l]}$ and $\phi(z^{(k)}) = y$. Let $z$ be a limit point of $\{z^{(k)}\}_{k \in \setN}$. Then $z_{(-\infty,-l]} = x_{(-\infty,-l]}$ and $\phi(z) = y$, as desired. Thus $\phi$ is right continuing with retract $l$. Similarly $\phi$ is left continuing with retract $l$.

    Next, assume that $\phi$ is \bict\ with bi-retract $n \in \setN$ and $Y$ is of finite type. We can assume that $Y$ is an edge shift and $\phi$ is a 1-block code.
    Suppose that $l \geq 0$ and $u \in \B_{2l+1}(X)$. Choose $x \in {}_{-l}[u]$ and let $y \in Y$ with $y_{[-l-n,l+n]} = \phi(x)_{[-l-n,l+n]}$. The point $\bar y$, given by ${\bar y}_i = \phi(x)_i$ for $i \leq 0$ and $ {\bar y}_i = y_i$ for $i \geq 0$, is in $Y$. Since $n$ is a right continuing retract and $\phi(x)_{(-\infty,l+n]} = {\bar y}_{(-\infty,l+n]}$, there is $z \in X$ such that ${z}_{(-\infty,l]} = {x}_{(-\infty,l]}$ and $\phi(z) = \bar y$. Then $\phi(z)_{[-l-n,\infty)} = {\bar y}_{[-l-n,\infty)} = y_{[-l-n,\infty)}$. Since $n$ is also a left continuing retract, there is an $\bar x \in X$ such that ${\bar x}_{[-l,\infty)} = {z}_{[-l,\infty)}$ and $\phi(\bar x) = y$. Note that ${\bar x}_{[-l,l]} = {z}_{[-l,l]} = {x}_{[-l,l]}$, hence $\bar x$ is in ${}_{-l}[u]$. So $_{-l-n}[\phi(x)_{[-l-n,l+n]}] \subset \phi(_{-l}[u])$. Thus $\phi$ is open with uniform lifting length $n$.
\end{proof}

\begin{rem} \label{rem:converse_still_hold}
    Note that the proof of the second statement in Lemma \ref{lem:open_ull_then_bict_and_converse} is still valid under a weaker assumption on $Y$, that is, when there is a \SFT\ $Z$ such that $\phi(X) \subset Z \subset Y$ and $Z$ is open and closed in $Y$.
\end{rem}

\begin{lem}\cite{Jung} \label{prop:factor_of_SFT_is_SFT}
    Let $X$ be a \SFT\ and $\phi$ an open factor code from $X$ to a shift space $Y$. Then $Y$ is of finite type.
\end{lem}

In \cite{BoyT}, the first statement of the following proposition is proved when both shift spaces are of finite type. We show that the codomain need not be of finite type.

\begin{prop} \label{prop:rct_from_SFT_has_retract_and_open_from_SFT_has_ull}
    Let $\phi \!: \! X \to Y$ be a code between shift spaces and $X$ of finite type.
    \begin{enumerate}
        \item If $\phi$ is \rct, then it has a retract.
        \item If $\phi$ is open, then it has a uniform lifting length.
    \end{enumerate}
\end{prop}
\begin{proof}

    By recoding, we can assume that $X$ is an edge shift and $\phi$ is 1-block.

    (1) It suffices to show the case where $\phi$ is onto. We claim that there exists $n \in \setN$ such that whenever $x \in X$ and $y \in Y$ with $\phi(x)_{(-\infty,0]} = y_{(-\infty,0]}$, we can find $\bar x \in X$ such that $\phi(\bar x)_{(-\infty,1]} = y_{(-\infty,1]}$ and $x_{(-\infty,-n]} = \bar x_{(-\infty,-n]}$.
    Suppose not. Then for each $k \in \setN$, there exist $x^{(k)} \in X$ and $y^{(k)} \in Y$ such that $\phi(x^{(k)})_{(-\infty,0]} = y^{(k)}_{(-\infty,0]}$ and there is no $\bar x \in X$ with $\phi(\bar x)_{(-\infty,1]} = y^{(k)}_{(-\infty,1]}$ and $\bar x_{(-\infty,-k]} = x_{(-\infty,-k]}$. By choosing a subsequence, we can assume that there are $x \in X$ and $y \in Y$ with $x^{(k)} \to x$, $y^{(k)} \to y$.
    Since $\phi$ is \rct\ and $\phi(x)_{(-\infty,0]} = y_{(-\infty,0]}$, there exist $z \in X$ and $m \geq 0$ such that $\phi(z) = y$ and $z_{(-\infty,-m]} = x_{(-\infty,-m]}$. Take $k > m$ satisfying $x^{(k)}_{[-m,0]} = x_{[-m,0]}$ and $y^{(k)}_{[-m,1]} = y_{[-m,1]}$.
    Define $\bar x \in X$ by letting
    \[
        {\bar x}_i =
            \begin{cases}
                (x^{(k)})_i     & \text{if $i \leq -m$}   \\
                z_i     & \text{if $i \geq -m$}
            \end{cases}
    \]
    Then $\phi(\bar x)_{(-\infty,1]} = y^{(k)}_{(-\infty,1]}$ and $\bar x_{(-\infty,-k]} = x^{(k)}_{(-\infty,-k]}$, which is a contradiction. Thus the claim holds.

    This $n$ in the claim is indeed a retract. Suppose $x \in X$ and $y \in Y$ satisfy $\phi(x)_{(-\infty,0]} = y_{(-\infty,0]}$. Let $z^{(0)} = x$. By inductive process, for each $k \in \setN$ there exists $z^{(k)} \in X$ such that $\phi(z^{(k)})_{(-\infty,k]} = y_{(-\infty,k]}$ and $z^{(k)}_{(-\infty,k-n-1]} = z^{(k-1)}_{(-\infty,k-n-1]}$. Since $z^{(k)}_{(-\infty,-n]} = x_{(-\infty,-n]}$ for all $k \in \setN$, any limit point $z$ of $\{z^{(k)}\}_{k \in \setN}$ satisfies $z_{(-\infty,-n]} = x_{(-\infty,-n]}$ and $\phi(z) = y$.

    (2) Choose $l \geq 0$ so that $\phi([a])$ consists of central $2l+1$ cylinders for all $a \in \B_1(X)$. Let $x \in X$ and $y \in Y$ satisfy $\phi(x)_{(-\infty,0]} = y_{(-\infty,0]}$. Consider $\phi({}_{-l}[x_{-l}])$. Since it contains $\phi(x)$ and $\phi(x)_{[-2l,0]}=y_{[-2l,0]}$, there exists $z \in {}_{-l}[x_{-l}]$ with $\phi(z) = y$. Now define $\bar x$ by ${\bar x}_i = {x}_i$ for $i \leq -l$ and ${\bar x}_i = {z}_i$ for $i \geq -l$. Then $\bar x$ is left asymptotic to $x$ and $\phi(\bar x) = y$, so $\phi$ is right continuing with retract $l$. Similarly $\phi$ is left continuing with retract $l$.

    Since $\phi$ is open, $Z = \phi(X)$ is open and closed in $Y$. Also $Z$ is a \SFT\ by Lemma \ref{prop:factor_of_SFT_is_SFT}. Thus by Remark \ref{rem:converse_still_hold}, $\phi$ is open with a uniform lifting length.
\end{proof}

\begin{rem}
    A code $\pXtoY$ between shift spaces is called \emph{\rct} a.e. (almost everywhere) if whenever $x \in X$ is left transitive in $X$ and $\phi(x)$ is left asymptotic to a point $y \in Y$, then there exists $\bar x \in X$ such that $\bar x$ is left asymptotic to $x$ and $\phi(\bar x) = y$. Similarly we have the notions called \lct\ a.e. and \bict\ a.e.
    A slight modification of the proof of Lemma \ref{prop:rct_from_SFT_has_retract_and_open_from_SFT_has_ull} (2) shows that an open code from a synchronized system is \bict\ a.e.
\end{rem}

\begin{thm}\label{thm:open_iff_bict}
    Let $\phi$ be a factor code from a \SFT\ $X$ to a sofic shift $Y$. Then $\phi$ is open \ifff\ it is \bict.
\end{thm}
\begin{proof}
    Suppose first that $\phi$ is open. By Lemma \ref{prop:factor_of_SFT_is_SFT}, $Y$ is of finite type. It follows from Proposition \ref{prop:rct_from_SFT_has_retract_and_open_from_SFT_has_ull} (2) and Lemma \ref{lem:open_ull_then_bict_and_converse} that $\phi$ is \bict.

    Suppose $\phi$ is \bict. Since $Y$ is sofic, there are a \SFT\ $Z$ and a factor code $\pi : Z \to Y$. Consider the fiber product $(\Sigma,\psi_1,\psi_2)$ of $(\phi,\pi)$ \cite{LM}. Then $\Sigma$ is of finite type. By a usual fiber product argument, it is easy to show that $\psi_2$ is also \bict. Since $\Sigma$ and $Z$ are of finite type, it follows from Proposition \ref{prop:rct_from_SFT_has_retract_and_open_from_SFT_has_ull} (1) and Lemma \ref{lem:open_ull_then_bict_and_converse} that $\psi_2$ is open. Since fiber product pulls down the openness of a code, it follows that $\phi$ is open \cite[Lemma 2.4]{Jung}.
\end{proof}

\vspace{0.02cm}
\begin{rem}
    Recently, J. Yoo \cite{Yoo} showed us that a \rct\ factor of a \SFT\ is also of finite type. This result, combined with Proposition \ref{prop:rct_from_SFT_has_retract_and_open_from_SFT_has_ull} (1) and Lemma \ref{lem:open_ull_then_bict_and_converse}, implies the `if' part of Theorem \ref{thm:open_iff_bict}.
\end{rem}

\vspace{0.02cm}
\begin{rem}
    It is well known that a \fto\ factor code between \rSFTs\ is open \ifff\ it is \bic\ \cite{Nas}. Thus Theorem \ref{thm:open_iff_bict} can be thought as an infinite-to-one version of this fact. In the \fto\ case, we also have the following relation: A \fto\ factor code between \rSFTs\ is open exactly when it is \cto. By virtue of this relation, it is natural to ask whether an \ito\ code between \rSFTs\ is open \ifff\ it is \ito\ everywhere, i.e., $\preimageAbs{\phi}{y} = \infty$ for all $y$. However, it turns out that each implication is not true.
\end{rem}

\vspace{0.02cm}
We summarize the implications between factor codes in the following diagram. When $X$ is of finite type, then all conditions but \bict\ a.e. are equivalent by Proposition \ref{prop:rct_from_SFT_has_retract_and_open_from_SFT_has_ull} and Theorem \ref{thm:open_iff_bict}. We present examples to show that in general these conditions are different.

\vspace{0.1cm}
\begin{table}[h]
$$\xymatrix{
    \txt{\textsf{open with a} \\ \textsf{uniform lifting length}} \ar@{=>}[dd]^{\text{(a)}} \ar@{=>}[rr]  & & \txt{\textsf{\bict\ with} \\ \textsf{a bi-retract}} \ar@{=>}[dd]_{\text{(b)}} \ar@/_1.5pc/@{:>}[ll]_{\text{$Y$=SFT}} \\
    & & \\
    \textsf{open} \ar@{=>}[rdd]_{\text{$X$=synchronized} \quad }^{\text{(c)}}   & &\textsf{\bict} \ar@{=>}[ldd]_{\text{(d)}}  \\
    & & \\
    & \textsf{\bict\ a.e.} &
}$$
\end{table}

\vspace{0.05cm}
\begin{exam} \label{ex:examples_section2}
    (1) Let $X = \X_{W_1}$ and $Y = \X_{W_2}$, where $W_1 = \{ ab^k c^k : k \geq 0\}$ and $W_2 = \{ a, abc \}$.
    Define $\Phi : \B_3(X) \to \A(Y)$ as follows: $\Phi(abc) = b, \Phi(bca) = c$ and  $\Phi(u) = a$ otherwise. Let $\phi : X \to Y$ be a code defined by $\phi(x)_i = \Phi(x_{[i-1,i+1]})$. Note that $\phi$ replaces each $b^k c^k, k > 1$ with $a^{2k}$ and also $b^\infty$, $c^\infty$ with $a^\infty$.

    First we show that $\phi$ is not \rct. Take $x = b^\infty \in X$ and $y = a^\infty . (abc)^\infty \in Y$. If $z \neq b^\infty$ is left asymptotic to $x$, then $z$ is in the orbit of $b^\infty.c^\infty$, and
    we have $\phi(z) = a^\infty \neq y$. Thus $\phi$ is not \rct. Similarly it is not \lct.

    Next we show that $\phi$ is open. Let $U$ be an open set of $X$ and $x \in U$. Since $U$ is open, there is $l \in \setN$ with $V = {}_{-l}[x_{-l} \cdots x_l] \subset U$. We claim that there is $\bar x \in V$ such that $\phi(\bar x) = \phi(x)$ and there are $i \leq -l$ and $j \geq l$ with $\bar x_i = \bar x_j = a$.
    If $x$ already has this property, we are done. Otherwise, there are two cases (up to symmetry):

    \vspace{0.1cm}
    \textsf{Case} 1. There is $i \leq -l$ with $x_i = a$ and $x_j \neq a$ for all $j \geq l$. Choose $\bar j \in \setZ$ such that $x_{\bar j} = a$ is the rightmost occurrence of $a$ in $x$. Let $\bar x = x_{(-\infty,l]} b^2 c^{l - \bar j + 2} a^\infty$. Since $x_{(\bar j, \infty)}$ must be of the form $b^\infty$, it follows that $\bar x \in V$ and $\phi(\bar x) = \phi(x)$.

    \vspace{0.1cm}
    \textsf{Case} 2. $a$ does not occur in $x_{(-\infty,-l]}$ and $x_{[l,\infty)}$. Suppose first that $a$ occurs in $x_{(-l,l)}$. Let $\bar i$ and $\bar j$ be the leftmost and rightmost occurrence of $a$ in $x$, respectively. Then $\bar x = a^\infty b^{\bar i - l + 2} c^2 x_{[-l,l]} b^2 c^{l - \bar j + 2} a^\infty \in X$ is the desired point. Otherwise, $x$ must be one of the following form: $b^\infty$, $c^\infty$, or $b^\infty c^\infty$. If $x = b^\infty$, then $\bar x = a^\infty b^l . b^{l+1} c^{2l+1} a^\infty$ satisfies the property. Other cases can be proved similarly. This proves the claim.

    \vspace{0.1cm}
    Take $i \leq -l$ and $j \geq l$ such that $\bar x_i = \bar x_j = a$. For each $y \in Y$ with $y_{[i,j]} = \phi(x)_{[i,j]}$, let $z = y_{(-\infty,i)}\bar x_{[i,j]}y_{(j,\infty)}$. Then we have $z \in V$ and $\phi(z) = y$. So ${}_i [\phi(x)_{[i,j]}] \subset \phi(U)$. Thus $\phi$ is an example of an open code which is neither right nor left continuing. By Lemma \ref{lem:open_ull_then_bict_and_converse}, $\phi$ has no uniform lifting length. Thus the converse of (a) does not hold in general.

    \vspace{0.2cm}
    (2) This example is due to J. Yoo \cite{Yoo}. Let $X$ be a shift space on the alphabet $\{ 1, \bar 1, 2, 3\}$ defined by forbidding $\{ \bar 1 2^n 3 : n \geq 0\}$, and $Y$ the full 3-shift $\{1,2,3\}^\setZ$. Define $\pXtoY$ by letting $\phi(\bar 1) = 1$ and $\phi(a) = a$ for all $a \neq \bar 1$.

    We first show that $\phi$ is \rct: Suppose $x \in X$ and $y \in Y$ satisfy $\phi(x)_{(-\infty,0]} = y_{(-\infty,0]}$. If $x^{(1)} = x_{(-\infty,-1]}.x_0 y_{[1,\infty)}$ is in $X$, then we are done. Otherwise, the word $\bar 1 2^n 3$ occurs exactly once in $x^{(1)}$. Let $x^{(2)}$ be the point obtained from $x^{(1)}$ by replacing $\bar 1 2^n 3$ with $1 2^n 3$. Then $\phi(x^{(2)}) = y$ and $x^{(2)}$ is left asymptotic to $x$. But considering $x = \bar 1^\infty 2^n . 22^\infty$ and $y = 1^\infty 2^n . 23^\infty$ for each $n \in \setN$, one can see that $\phi$ has no (right continuing) retract.

    Now we show that $\phi$ is left continuing with a (left continuing) retract. If $x \in X$ and $y \in Y$ satisfy $\phi(x)_{[0,\infty)} = y_{[0,\infty)}$, then by letting $\bar x = y_{(-\infty,-1]}.x_{[0,\infty)}$, we have $\bar x \in X$, $\phi(\bar x) = y$ and $\bar x_{[0,\infty)} = x_{[0,\infty)}$. Thus $\phi$ is \lct\ with retract 0, hence it is \bict\ but does not have a right continuing retract. So a continuing code may not have a retract and the converse of (b) does not hold.

    Moreover, $\phi$ is an example of a \bict\ code which is not open. For, if $\phi$ is open, then $\phi([\bar 1]) = \bigcup_{i=1}^{k} C_i$ for central $2l+1$ cylinders $C_i$ in $Y$. Since $1^\infty.12^\infty \in \phi([\bar 1])$, there exists $C_i$ of the form ${}_{-l}[1^{l+1}2^l]$. Let $y=1^\infty.12^l3^\infty \in C_i$. Then $y \notin \phi([\bar 1])$, a contradiction comes. Indeed $X$ is an irreducible strictly sofic shift.

    \vspace{0.2cm}
    (3) As is easily seen, if $Y$ is the even shift and $\phi : \X_A \to Y$ is its canonical cover (i.e., a code given by its minimal right resolving presentation \cite{LM}), then $\phi$ is \bict\ a.e. but neither open nor \bict. Thus the converses of (c) and (d) are false even if the domain is of finite type.
\end{exam}

\vspace{0.3cm}
\section{Almost specified shifts}\label{Section_III:Almost_Specifications}

To investigate the existence of continuing codes, we consider the almost specified shifts. In this section we present several properties of almost specified shifts defined as follows. These will be used in subsequent sections.

\begin{defn} \label{def:specification+almost_specification}
    Let $X$ be a shift space.
    \begin{enumerate}
        \item \label{def:specification} $X$ has the \emph{specification property} if there exists $N \in \setN$ such that for all $u,v \in \B(X)$, there exists $w \in \B_N(X)$ with $uwv \in \B(X)$.
        \item \label{def:almost_specification} $X$ has the \emph{almost specification property} (or $X$ is \emph{almost specified}) if there exists $N \in \setN$ such that for all $u,v \in \B(X)$, there exists $w \in \B(X)$ with $uwv \in \B(X)$ and $|w| \leq N$.
    \end{enumerate}
\end{defn}

Note that an \rSofic\ is \as.

\begin{lem} \cite{Ber} \label{lem:almost_specified_then_synchronized}
    If $X$ is an \ass, then it is a synchronized system.
\end{lem}
\begin{proof}
    Let $N$ satisfy the condition in Definition \ref{def:specification+almost_specification} (\ref{def:almost_specification}). For given $u,v \in \B(X)$, let $B(u,v)$ be the set of words $w$ such that $uwv \in \B(X)$ and $|w| \leq N$. Note that $B(u,v)$ is a nonempty finite set. Fix two words $u^{(0)}$ and $v^{(0)}$ in $\B(X)$. We define $u^{(n)}$ and $v^{(n)}$ inductively as follows.
    Suppose that $u^{(0)}, \cdots, u^{(n-1)}$ and $v^{(0)}, \cdots, v^{(n-1)}$ are given. If there are $u, v \in \B(X)$ such that
    $$\emptyset \neq B(u u^{(n-1)} \cdots u^{(0)}, v^{(0)} \cdots v^{(n-1)} v) \varsubsetneq B(u^{(n-1)} \cdots u^{(0)}, v^{(0)} \cdots v^{(n-1)}),$$
    then let $u^{(n)} = u$ and $v^{(n)} = v$. Since $B(u,v)$'s are nonempty finite sets, this process eventually terminates in the sense that there is $m \in \setN$ such that by letting $\bar u = u^{(m)} \cdots u^{(0)}$ and $\bar v = v^{(0)} \cdots v^{(m)}$ we have $B(u \bar u, \bar v v) = B(\bar u, \bar v)$ for all $u, v \in \B(X)$ with $u \bar u, \bar v v \in \B(X)$. Then $\bar u w \bar v$ is a synchronizing word for each $w \in B(\bar u, \bar v)$.
\end{proof}

The following lemma shows that the \aspe\ is a natural irreducible version of the \spe.

\begin{lem} \label{lem:spe_iff_aspe}
    Let $X$ be a mixing shift space. Then $X$ has the \aspe\ \ifff\ it has the \spe.
\end{lem}
\begin{proof}
    Suppose $X$ is \as.
    Let $N$ be given as in Definition \ref{def:specification+almost_specification} (2) and by using Lemma \ref{lem:almost_specified_then_synchronized} take a synchronizing word $w \in \B(X)$. Since $X$ is mixing, there is $L > 0$ such that for all $l \geq L$, we can find $u \in \B_l(X)$ with $wuw \in \B(X)$. Let $M = 2N + 2|w| + L$. For any $u, v \in \B(X)$, there exist $\bar u, \bar v \in \B(X)$ with $|\bar u|, |\bar v| \leq N$ and $u \bar u w, w \bar v v \in \B(X)$. Also we can find $\bar w \in \B_{L+2N-|\bar u|-|\bar v|}(X)$ with $w\bar w w \in \B(X)$. Let $\tilde w = \bar u w \bar w w \bar v$. Then $u \tilde w v \in \B(X)$ and $|\tilde w| = M$. Thus $M$ satisfies the condition in Definition \ref{def:specification+almost_specification} (\ref{def:specification}) and $X$ has the \spe.
\end{proof}

For a shift space $X$, denote by $X^-$ ($X^+$, resp.) the set of left (right, resp.) infinite sequences of $X$. For $x \in X$, we denote by $x^-$ (resp. $x^+$) the sequence $x_{(-\infty,-1]} \in X^-$ (resp. $x_{[0,\infty)} \in X^+$). If $x^- \in X^-$ and $u \in \B(X)$, their concatenation $x^- u$ is naturally defined. We write $x^- u \in X^-$ when there exists a sequence $y \in X^+$ such that $x^- u y \in X$. A similar operation is defined on $\B(X)$ and $X^+$. Then we have the following lemma whose proof is obvious by compactness.

\begin{lem} \label{lem:spe_aspe_iff_connectable}
    Let $X$ be a shift space.
    \begin{enumerate}
        \item $X$ has the \spe\ \ifff\ there is $N \in \setN$ such that when- ever $x^- \in X^-$ and $u \in \B(X)$, there is $w \in \B_N(X)$ such that $x^- w u \in X^-$.
        \item $X$ is \as\ \ifff\ there is $N \in \setN$ such that whenever $x^- \in X^-$ and $u \in \B(X)$, there is $w \in \B(X)$ such that $|w| \leq N$ and $x^- w u \in X^-$.
    \end{enumerate}
\end{lem}

\begin{exam} \label{ex:Sgap}
    Let $S = \{ n_1, n_2, \ldots \} \subset \Zplus$ with $n_i < n_{i+1}$ (possibly, finite) and $X = \X(S)$ the coded system generated by $\{ 10^n : n \in S \}$.
    This $\X(S)$ is called the $S$-gap shift \cite{LM}. Note that $X$ is a synchronized system. Then we have the following characterizations.
    \begin{enumerate}
        \item $X$ is \as\ \ifff\ $\sup_i |n_{i+1} - n_i| < \infty$.
        \item $X$ is mixing \ifff\ $\gcd\{n + 1 : n \in S \} = 1$.
        \item $X$ has the \spe\ \ifff\ $\sup_i |n_{i+1} - n_i| < \infty$ and $\gcd\{n + 1 : n \in S \} = 1$.
    \end{enumerate}
\end{exam}
\begin{proof}
    (1) If $S$ is finite, then $X$ is of finite type (hence almost specified) and $\sup_i |n_{i+1} - n_i| < \infty$, thus the equivalence is clear. So we may assume $S$ is infinite. Suppose that $X$ is \as. Let $N$ be given as in Definition \ref{def:specification+almost_specification} (\ref{def:almost_specification}). Then for each $n_i \in S$, there exists $w$ such that $10^{n_i + 1} w 1 \in \B(X)$ and $|w| \leq N$. Thus $ |n_{i+1} - n_{i}| \leq N+1$ for all $i$. Conversely, if $L = \sup_i |n_{i+1} - n_i| < \infty$, then let $N = \max(n_1,L)$. For any $v, w \in \B(X)$, there exist $i, j \leq N$ with $v0^i 1, 1 0^j w \in \B(X)$. Then $v0^i 1 0^j w \in \B(X)$ and $|0^i 1 0^j| \leq 2N+1$. Thus $X$ is \as.

    (2) If $X$ is mixing, then there exists $N > 0$ such that for all $n \geq N$, there exists $w \in \B_n(X)$ with $1w1 \in \B(X)$. Thus there are words of length $N+1$ and $N+2$ of the form $10^{i_1}10^{i_2} \ldots 10^{i_m}$ with $i_k \in S$, implying that $\gcd\{n + 1 : n \in S \} = 1$.

    On the other hand, if $\gcd\{n + 1 : n \in S \} = 1$ then for all sufficiently large $n$, there is a word of length $n$ of the form $10^{i_1}10^{i_2} \ldots 10^{i_m}1$. Since $1$ is synchronizing and $X$ is irreducible, the result follows.

    (3) This follows from (1), (2) and Lemma \ref{lem:spe_iff_aspe}.
\end{proof}

\begin{exam}
    We list some examples reflecting strict inclusions between classes of shift spaces. \begin{enumerate}
        \item Let $X = \X(S)$ be the $S$-gap shift, where $$ S = ( 2 \setN + 1 ) \setminus \biggl[ \bigcup_{k=1}^{\infty} \big( k\setN \cap (10^{k-1} , 10^k )  \big) \biggr].$$
            Then it is a nonsofic \as\ shift without the specification property.
        \item For any $X$ with the \spe, define its root $\widetilde X$ by
            $$\widetilde X = \set{ ({\bar x}_i)_{i \in \setZ} : \exists x \in X \text{ with } {\bar x}_{2i+1} = a, {\bar x}_{2i} = x_i \text{ for all } i \in \setZ},$$
            where $a$ is a symbol which is not in $\B_1(X)$. Then $Y = \widetilde X \cup \sigma (\widetilde X)$ is \as, but does not have the \spe.
        \item Let $W = \{ ab^k c^k : k \geq 1\}$ and $X = \X_W$. It is easy to see that $X$ is mixing and synchronized. But if $x^- = b^\infty$, then there is no $w \in \B(X)$ such that $x^- w a \in X^-$. By Lemma \ref{lem:spe_aspe_iff_connectable}, $X$ cannot be \as.
    \end{enumerate}
\end{exam}

A synchronized system has a global periodic structure similar to an \rSFT. The collection of the subsets $D_i$ in the following theorem is called \emph{the cyclic cover of $X$}. It is unique up to cyclic permutation.

\begin{thm}\label{thm:cyclic_cover}\cite{Tho}
    Let $X$ be a synchronized system. Then there exist a unique $p \in \setN$ and closed sets $D_i \subset X$, $i = 0, 1, \cdots, p-1$, such that
    \begin{enumerate}
        \item[i)] $X = \bigcup_{i=0}^{p-1} D_i$,
        \item[ii)] $\sigma(D_i) = D_{i+1 \modp}$,
        \item[iii)] $\sigma^p|_{D_i}$ is mixing for all $i = 0 , 1, \cdots, p-1$, and
        \item[vi)] $D_i \cap D_j$ has empty interior when $i \neq j$.
    \end{enumerate}
\end{thm}

If $X$ is an \rSFT, then this decomposition is well known (cf, \cite{LM}). In this case $p = \per(X)$ and $D_i \cap D_j = \emptyset$ for $i \neq j$.

Let $X^p$ be the $p$th higher power shift of $X$ and $\gamma : X \to X^p$ the $p$th higher power code \cite{LM} given by $\gamma(x)_i = x_{[ip,ip+p-1]}.$
If $u \in \B_{pM}(X)$, then $\gamma(u) \in \B_M(X^p)$ is naturally defined.
This $\gamma$ is a topological conjugacy between $(X,\sigma^p)$ and $(X^p,\sigma)$. When $X$ is a synchronized system, we define $X_i = \gamma(D_i)$. Then $X_i$ is an irreducible subshift of $X^p$ and $\gamma$ is also a topological conjugacy between $(D_i, \sigma^p)$ and $(X_i,\sigma)$ for each $0 \leq i < p$. Note that for a block $u \in \B_{pM}(X)$, $\gamma(u) \in \B(X_i)$ \ifff\ there exists a point $x \in D_i$ with $x_{[0,pM-1]} = u$.
From the construction of the cyclic cover of $X$, it follows that the sets
$$\{ \gamma(u) \in \B(X_i) : u \text{ is a synchronizing word for } X \}$$
are disjoint. This is because $X_i$ is the closure of an `\emph{irreducible subshift}' of $X^p$ in the sense of Thomsen (see \cite[\S 3]{Tho}). Thus for a synchronized word $u \in \B_{pM}(X)$, there exists a unique $i$ with $\gamma(u) \in \B(X_i)$. We use this fact to prove the following result.

\begin{lem}\label{lem:aspe_then_cover_has_specification}
    Let $X$ be an \ass. If $\{D_0, \cdots, D_p \}$ is the cyclic cover of $X$, then $\sigma^p|_{D_i}$ has the \spe.
\end{lem}
\begin{proof}
    It suffices to show the case $i=0$. Note that $\sigma^p|_{D_0}$ has the \spe\ \ifff\ $X_0$ has the \spe.
    Since $X$ is \as, there exists $N$ satisfying the condition in Definition \ref{def:specification+almost_specification} (\ref{def:almost_specification}). We claim that $X_0$ is also \as.

    Suppose that $\tilde u = \gamma(u), \tilde v = \gamma(v) \in \B(X_0)$. By extending $\tilde u$ to the left and $\tilde v$ to the right, we may assume that $u, v$ are synchronizing words for $X$. Note that $|u|, |v|$ are multiples of $p$. Since $X$ is \as, there exists a word $w \in \B(X)$ with $l = |w| \leq N$ such that $uwv \in \B(X)$. Let $x \in X$ be a point with $x_{[0,|uwv|-1]} = uwv$. Since $u$ is synchronizing and $\gamma(u) = \tilde u \in \B(X_0)$, we have $\gamma(x) \in X_0$ (by the remark following Theorem \ref{thm:cyclic_cover}). Thus $\gamma(\sigma^{|u|+l}(x)) \in X_{l \modp}$, and we get $\tilde v \in \B(X_{l \modp})$. Since $v$ is synchronizing and $\tilde v \in \B(X_0)$, it follows that $l = 0 \modp$ and we have $$\tilde u \gamma(w) \tilde v = \gamma(uwv) \in \B(X_0).$$
    Thus $X_0$ is \as, as desired. Since $X_0$ is mixing, the result follows from Lemma \ref{lem:spe_iff_aspe}.
\end{proof}

\begin{prop} \label{prop:aspe_then_SFT_exists}
    Let $X$ be an \ass\ with $h(X) > 0$. For given $\epsilon > 0$, there exist an \rSFT\ $Z \subset X$ and $k \in \setN$ such that $h(Z) > h(X) - \epsilon$ and any $k$-block of $Z$ is a synchronizing word of $X$. If $X$ is mixing, then $Z$ can be chosen to be mixing.
\end{prop}
\begin{proof}
    The proof essentially follows the lines in \cite{Boy84, Mar}. First, we prove the case where $X$ has the \spe. Using Lemma \ref{lem:almost_specified_then_synchronized} find a synchronizing word $w \in \B(X)$. For each $k > |w|$, let $X_k$ be the $(k-1)$-step \SFT\ whose $k$-blocks are $a_1 \ldots a_k \in \B_k(X)$ such that $w$ occurs in $a_2 \ldots a_k$. Then $X_k \subset X$.

    Let $N$ be given in Definition \ref{def:specification+almost_specification} (\ref{def:specification}). Fix $k > 2N + 2|w|$. For each $l = mk$ with $m \in \setN$, we have $$|\B_l(X_k)| \geq |\B_{k-2N-2|w|}(X)|^m,$$
    since for each $m$ pair of words $v^{(1)}, \cdots, v^{(m)} \in \B_{k-2N-2|w|}$, there exists a word in $\B_l(X_k)$ whose initial sequence is of the form $w u^{(1)} v^{(1)} \bar u^{(1)} w u^{(2)} v^{(2)} \bar u^{(2)} w \cdots v^{(m)} \bar u^{(m)} w$, with $u^{(i)}, \bar u^{(i)} \in \B_N(X)$.
    It follows that $h(X_k) \geq \frac{1}{k} \log |\B_{k-2N-2|w|}(X)|$. Thus $\lim_{k \to \infty} h(X_k) = h(X)$. It is easy to see that $X_k$ is mixing for all large $k$. Take $k$ large so that $X_k$ is mixing, $ h(X_k) > h(X) - \epsilon$ and let $Z = X_k$.

    Next, we prove the case where $X$ is \as. Let $\{D_0, \cdots, D_p \}$ be the cyclic cover of $X$ and take $X_0 = \gamma(D_0)$ as in the remark following Theorem \ref{thm:cyclic_cover}. Since $X_0$ has the \spe\ by Lemma \ref{lem:aspe_then_cover_has_specification}, by applying the previous result to $X_0$, we get $Z_0 \subset X_0$ satisfying the properties. Define $$Z = \bigcup_{i=0}^{p-1} \sigma^i (\gamma^{-1}(Z_0)).$$
    Then $Z$ satisfies all the desired properties.
\end{proof}
\begin{rem}
    There is a synchronized system $X$ such that $\sup_{Y \subset X} h(Y) < h(X)$, where the supremum is taken over all sofic subshifts $Y$ of $X$ \cite{Pet}. Thus Proposition \ref{prop:aspe_then_SFT_exists} does not hold when $X$ is merely synchronized.
\end{rem}

\vspace{0.3cm}

\section{Extension Theorem: A mixing case}\label{sec:main}

In this section, we prove that when $X$ is \as\ and $Y$ is mixing and of finite type, the entropy and the periodic conditions guarantee the existence of a \bict\ factor code from $X$ to $Y$ with a bi-retract. With little extra work, in fact we prove further that every code from a proper subshift of $X$ to $Y$ can be extended to an open \bict\ code. First we need some lemmas.

\begin{lem}\label{lem:exist_bic}
    Let $X$ and $Y$ be \rSFTs\ with $h(X) > h(Y)$. Then there exist an \rSFT\ $Z \subset X$ and a \bic\ factor code $\pi : Z \to Y$.
\end{lem}
\begin{proof}
    We can assume $Y = \X_B$ with $B$ irreducible. Let $m$ be the size of $B$. For each $n \in \setN$, define $Y_n = \X_{B_n}$, where $B_n$ is the $nm \times nm$ irreducible matrix given by
    $$B_n =
        \begin{pmatrix}
            0 & B & 0 & \cdots & 0 \\
            0 & 0 & B & \cdots & 0 \\
            \vdots & \vdots & \vdots & \ddots & \vdots \\
            0 & 0 & 0 & \cdots & B \\
            B & 0 & 0 & \cdots & 0
        \end{pmatrix}.$$
    Note that $\per(Y_n) = n \cdot \per(Y)$ and $p_{kn}(Y_n) = n \cdot p_{kn}(Y)$ for each $k \in \setN$.
    For each $n \in \setN$, it is easy to see that the natural projection code from $Y_n$ to $Y$ is \bir\ (e.g., see \cite{Nas}). So it suffices to show that $Y_n$ embeds into $X$ for large $n$.

    Let $\epsilon = ({h(X) - h(Y)})/3$. Take $n$ large such that $\per(X) | n $, $n^{-1} \log n < \epsilon, $
    $$\frac{1}{kn} \log p_{kn}(Y) < h(Y) + \epsilon \text{ for all } k \in \setN, \text{and}$$
    $$\frac{1}{kn} \log q_{kn}(X) > h(X) - \epsilon \text{ for all } k \in \setN.$$
    Then for each $k \in \setN$,
    $$ \frac{1}{kn} \log n + \frac{1}{kn} \log p_{kn}(Y) < h(Y) + 2 \epsilon = h(X) - \epsilon < \frac{1}{kn} \log q_{kn}(X)$$
    so we have $n \cdot p_{kn}(Y) < q_{kn}(X),$ and therefore
    $$q_{kn}(Y_n) \leq p_{kn}(Y_n) = n \cdot p_{kn}(Y) < q_{kn}(X).$$
    Since $p_j(Y_n) = 0$ if $j \notin n\setN$, we have $q_j(Y_n) \leq q_j(X)$ for all $j \in \setN$. Now applying the Krieger's Embedding Theorem \cite{Kri} gives the result.
\end{proof}

\begin{lem}\cite[Theorem 26.17]{DGS} \label{lem:DGS}
    Let $X$ be a mixing \SFT\ and $\widetilde X$ a proper subshift of $X$. For given $h < h(X)$, there is a mixing \SFT\ $Z \subset X$ such that $h(Z) > h$ and $Z \cap \widetilde X = \emptyset$.
\end{lem}

\begin{lem}[Extension Lemma]\cite{Boy84}\label{lem:extension}
    Let $X$ be a shift space and $Y$ a mixing \SFT\ such that $P(X) \searrow P(Y)$. Then any code from a proper subshift of $X$ to $Y$ can be extended to a code from $X$ to $Y$.
\end{lem}

It is the following theorem which we extend in this section. By usual reduction to the mixing case, Theorem \ref{thm:BoyT} shows that the entropy and the periodic conditions guarantee the existence of a \rct\ factor code between two \rSFTs. However, to extend an arbitrary code on a proper subshift to a code on the whole domain, the mixing condition on $Y$ is crucial (see Example \ref{ex:extensionfail}).

\begin{thm} \cite{BoyT} \label{thm:BoyT}
    Let $X$ be an \rSFT\ and $Y$ a \mSFT\ such that $h(X) > h(Y)$ and $P(X) \searrow P(Y)$. Then any code from a proper subshift of $X$ to $Y$ can be extended to a \rct\ code on $X$.
\end{thm}

Now we are ready to prove the main theorem in this section.
As indicated in \S 1, we extend the above theorem in two directions. This extension result will be used in \S 5 to obtain our main results: Theorem \ref{thm:main} and Theorem \ref{thm:main2}.

\begin{thm} \label{thm:exist_bict_code_1}
    Let $X$ be an \ass\ and $Y$ a \mSFT\ such that $h(X) > h(Y)$ and $P(X) \searrow P(Y)$. Then any code from a proper subshift of $X$ to $Y$ can be extended to an open \bict\ code on $X$ with a bi-retract.
\end{thm}
\begin{proof}
    Suppose that $\widetilde X$ is a proper subshift of $X$ and $\tilde \phi : \widetilde X \to Y$ is a code. We will construct a \bict\ code $\pXtoY$ with a bi-retract such that $\phi|_{\widetilde X} = \tilde \phi$.
    We divide the proof into three parts. In Part I, we construct an extension code in the case where $X$ has the \spe. In Part II, we prove that the code constructed in Part I is \bict\ with a bi-retract. In Part III, by using the results in Part I and II we prove the case where $X$ is \as.

    \vspace{0.15cm}
    {\large \textsf{Part} I. \quad} In this part, we prove the case where $X$ has the \spe.
    By Proposition \ref{prop:aspe_then_SFT_exists}, there exists a \mSFT\ $Z_0$ with $h(Z_0) > h(Y)$ in which all sufficiently long blocks are synchronizing for $X$.
    By Lemma \ref{lem:DGS}, we can find a \mSFT\ $Z_1 \subset Z_0$ disjoint from $\widetilde X$ with $h(Z_1) > h(Y)$. Also by Lemma \ref{lem:exist_bic}, there exist an \rSFT\ $Z \subset Z_1$ and a biclosing factor code $\pi : Z \to Y$.
    By Lemma \ref{lem:extension}, we can find a factor code $\psi : X \to Y$ such that $\psi|_{Z} = \pi$ and $\psi|_{\widetilde X} = \tilde \phi$. Finally find a \mSFT\ $V \subset Z_1$ disjoint from $Z$ with $h(V) > h(Y)$ by using Lemma \ref{lem:DGS}.
    \begin{figure}[h]
        \includegraphics[width=6.0cm]{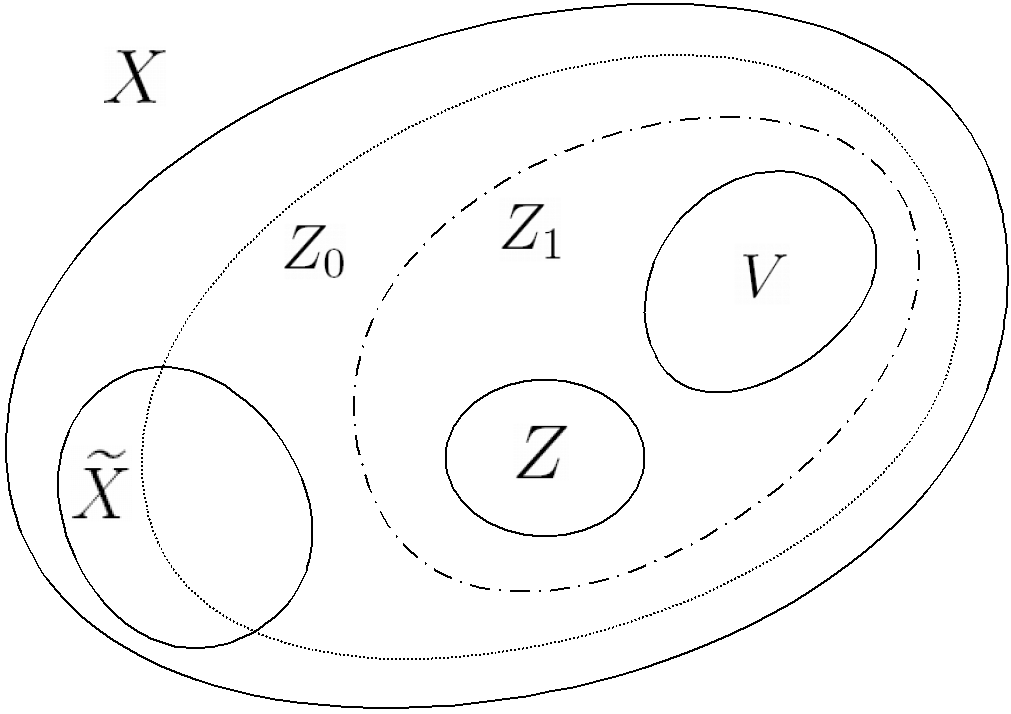}
    \end{figure}

    By passing to higher block shifts, we can assume that
    \begin{enumerate}
        \item[(a)] $Z_0, Z, V$ and $Y$ are edge shifts,
        \item[(b)] $\A(Z) \cap \A(\widetilde X) = \emptyset$ and $\A(Z) \cap \A(V) = \emptyset$,
        \item[(c)] $\psi$ is a 1-block code,
        \item[(d)] if $a, b \in \A(Z)$ and $ab \in \B(X)$, then $ab \in \B(Z)$,
        \item[(e)] each $a \in \A(Z_0)$ is a synchronizing word for $X$.
    \end{enumerate}
    Let $D$ be the right and \lc\ delay of $\pi$ \cite{LM}. Take $N$ large such that $N$ is a transition length for $X$, $Y$ and a weak transition length for $Z$.
    Fix a synchronizing word $\alpha \in \A(V)$. Since $X$ has the \spe, by Lemma \ref{lem:spe_aspe_iff_connectable} for each $x \in X$ there exist $w, \bar w \in \B_N(X)$ such that $x^- w \alpha \in X^-$ and $\alpha \bar w x^+ \in X^+$.

    Fix $i \gg 3N$. For each $a \in \A(X), b \in \A(Z)$ and $w \in \B_N(X)$, define
    \[ \begin{split}
            \HL_{i}(a;w;b) = \{ u \in \B_i^X(a,b) : & u_{[1,N+2]} = a w \alpha, u_{[N+2,i-2N)} \in \B(V), \\
            & u_{i-N} \notin \A(Z), \text{ and } u_{(i-N,i]} \in \B(Z) \};
       \end{split}
    \]
    \[ \begin{split}
            \LH_{i}(b;w;a) = \{ u \in \B_i^X&(b,a) : u_{[1,N]} \in \B(Z), u_{N+1} \notin \A(Z), \\
            & u_{(2N+1,i-N-1]} \in \B(V), \text{ and } u_{[i-N-1,i]} = \alpha w a \}.
       \end{split}
    \]
    (Note that there is no explicit restriction on subinterval $[i-2N,i-N)$ in the definition of $\HL_{i}(a;w;b)$. It guarantees that for a fixed $a \in \A(X)$ and $w \in \B_N(X)$ with $aw\alpha \in \B(X)$, there exists $b \in \A(Z)$ with $\HL_{i}(a;w;b) \neq \emptyset$. Similarly for $\LH_{i}(b;w;a)$.)
    Since $h(V) > h(Y)$, there is $I \in \setN$ such that
    \[  \begin{split}
            &|\HL_{I+N}(a;w;b)| \geq |\B_{I+N}^Y(\psi a,\psi b)| \text{ and } \\
            &|\LH_{I+N}(b;w;a)| \geq |\B_{I+N}^Y(\psi b,\psi a)|
        \end{split}
    \]
    for all $a$, $b$ and $w \in \B(X)$ whenever these sets are nonempty. (This is possible since if $\HL_{I+N}(a;w;b) \neq \emptyset$, then the asymptotic cardinality of this set is greater than  $ce^{(I+N)h(V)}$ for some $c > 0$ by an application of Perron-Frobenius Theory. Similarly for $\LH_{I+N}(b;w;a)$.)
    For each $a \in \B(X)$ and $b \in \A(Z)$, define surjections $\Psi_{HL}^{a,b}$ from $\B_{I+N}^X(a,b)$ onto $\B_{I+N}^Y(\psi a,\psi b)$ such that the restriction $\Psi_{HL}^{a,b} |_{\HL_{I+N}(a;w;b)}$ is surjective for each $w \in \B_N(X)$ with ${\HL_{I+N}(a;w;b)} \neq \emptyset$. Similarly define surjections $\Psi_{LH}^{b,a}$.

    Finally for each $2N \leq j \leq 2N + 2I$, define a map $\Phi_j : \A(X)^2 \to \B_j(Y)$ such that $\Phi_j(c,d) \in \B_j^Y(\psi c, \psi d)$. This is possible since $N$ is a transition length for $Y$. These maps will be used as marker fillers.

    For given $x \in X$, we divide $x \in X$ into low and high-stretches as in \cite{BoyT}. Call a segment of $x$ a \emph{low-stretch} if it is a $Z$-word of length $> 2N + D$, and not preceded or followed by a symbol from $\A(Z)$, i.e., a maximal $Z$-word of length $ > 2N + D$. Remaining stretches of maximal length are called \emph{high-stretches} (of $x$). By the condition (d), low-stretches of $x$ cannot overlap and hence $x$ is uniquely decomposed as low and high-stretches. Also, if a high-stretch of $x$ is of length greater than $2I$, then it is called a \emph{long} high-stretch. Otherwise, call it a \emph{short} high-stretch. The figure below shows a typical decomposition of $x$.

    \begin{figure}[h]
        \includegraphics[width=12cm]{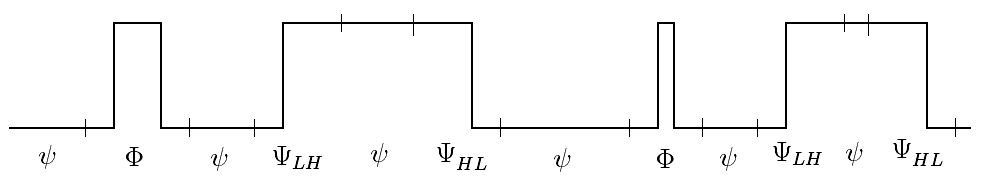}
    \end{figure}

    Now we define a code $\pXtoY$. Let $x \in X$.
    \begin{enumerate}
        \item[i)] \textit{low-stretches.} If $x_{[i-N,i+N]}$ is in a low-stretch, then define $\phi(x)_i = \psi(x_i)$.
        \item[ii)] \textit{long high-stretches.} If $x_{[i-I,i+I]}$ is in a long high-stretch, let $\phi(x)_i = \psi(x_i)$.
        \item[iii)] \textit{short high-stretches.} If $x_{[i,j]}$ is a short high-stretch, then $j-i+1 \leq 2I$. Define $\phi(x)_{[i-N,j+N]} = \Phi_{2N + j-i+1}(x_{i-N},x_{i+N})$.
        \item[iv)] \textit{high-low transition.} If $x_{[i,i+I)}$ is the end of some long high-stretch and $x_{[i+I,i+N+I)}$ is the beginning of some low-stretch, then define
            $$\phi(x)_{[i,i+N+I)} = \Psi_{HL}^{x_i,x_{i+N+I-1}} (x_{[i,i+N+I)}).$$
        \item[v)] \textit{low-high transition.} Similarly as in (iv), using $\Psi_{LH}$.
    \end{enumerate}

    The figure above describes the action of $\phi$ on parts of a point corresponding to some of these cases.
    Note that these cases cover all parts of $x$ and $\phi$ is a well-defined code from $X$ to $Y$. Indeed $\phi$ has memory and anticipation $2N+2I+D$. Since $x \in Z$ consists of a single low-stretch and $x \in \widetilde X$ consists of a single high-stretch, we have $\phi|_Z = \pi$ and $\phi|_{\widetilde X} = \tilde \phi$ and hence $\phi$ is a factor code which is an extension of $\tilde \phi$.

    \vspace{0.20cm}
    {\large \textsf{Part} II. \quad} We show that $\phi$ is \bict\ with bi-retract $n = 2I + 6N + 3D$. Suppose $x \in X$, $y \in Y$ satisfy $\phi(x)_{(-\infty,0]} = y_{(-\infty,0]}$.

    \vspace{0.1cm}
    \textsf{Case} 1. Suppose there exists an $i \in [-n,-2N-D]$ such that $x_{[i,i+2N+D]}$ is part of a low-stretch. Then, since $\pi$ is a \rc\ factor code, there exists a one-sided sequence $z_{[i+N,\infty)}$ in $Z^+$ such that $z_{i+N} = x_{i+N}$ and $\pi(z_{[i+N,\infty)}) = y_{[i+N,\infty)}$. Define a point ${\bar x}$ by
    $ {\bar x}_k = x_k$ if $k \leq i+N$, and $ {\bar x}_k = z_k$ if  $k \geq i+N$.
    Note that $\bar x \in X$ since $x_{i+N}$ is synchronizing. Also note that $\bar x_{[i,\infty)}$ is a part of low-stretch of $\bar x$ and therefore rule i) applies for $k \geq i+N$ and we have $\phi(\bar x) = y$.

    \textsf{Case} 2. If Case 1 does not holds, then $x_{[-n+2N+D,-2N-D]}$ is part of a high-stretch (note that there may be a part of a low-stretch at the left end of the interval). This interval is of length $2I + 2N + D$ and thus it is a part of a long high-stretch of $x$. Since there is no $Z$-word of length greater than $2N+D$ in this part, there exists $a \in \A(X) \setminus \A(Z)$ and $ -4N-2D-I \leq i \leq -2N-D-I$ with $x_i = a$ (by (d) in the recoding step).
    Since $X$ has the \spe, by Lemma \ref{lem:spe_aspe_iff_connectable} there exists $w \in \B_N(X)$ such that $x_{(-\infty,i]} w \alpha \in X^-$. Take $b \in \A(Z)$ with $\psi(b)=y_{i+I+N-1}$ and $\HL_{I+N}(a;w;b) \neq \emptyset$.

    \begin{claim}
        Such $b$ exists.
    \end{claim}

    \begin{proof}
        Let $p = \per(Z)$. Then there exists a partition $\{A_0, A_1, \cdots, A_{p-1} \}$ whose union is $\A(Z)$ with the property that whenever $ab \in \B(Z)$ and $a \in A_k$, we have $b \in A_{k+1 (\mathrm{mod}~ p)}$. Since $Y$ is mixing and $\pi : Z \to Y$ is a factor code, for each $k$ $\pi|_{A_k} : A_k \to \A(Y)$ is onto.
        % ...
        Take any block $v$ of length $I-N-1$ with $v_{[1,N+2]} = a w \alpha$ and $v_{[N+2,I-N)} \in \B(V)$.
        There exist $d \in \A(Z)$ and $c \in \A(Z_0) \setminus \A(Z)$ such that $cd \in \B(X)$. If we take $\bar k$ with $d \in A_{\bar k}$, then there exists $b \in A_{\bar k}$ such that $\pi(b) = y_{i+I+N-1}$. Since $N$ is a weak transition length for $Z$, there exists a block $u \in \B_N(Z)$ with $u_1 = d$ and $u_N = b$. As $Z_0$ is an edge shift containing $Z$, we have $cu \in \B(X)$.
        % ...
        Since $N$ is a transition length for $X$, we can find a block $w \in \B_N(X)$ with $vwcu \in \B(X)$. Then $vwcu \in \HL_{I+N}(a;w;b)$, which completes the proof.
    \end{proof}

    Since $\Psi_{HL}^{a,b}|_{\HL_{I+N}(a;w;b)}$ is onto, there exists $u_{[i,i+I+N)} \in \HL_{I+N}(a;w;b)$ with $\Psi_{HL}^{a,b}(u)$ $= y_{[i,i+I+N)}$. Since $u_{i+I+N-1} \in \A(Z)$ and $\pi$ is \rc, there exists $z_{[i+I+N-1,\infty)}$ in $Z^+$ such that $z_{i+I+N-1} = u_{i+I+N-1}$ and $\pi(z_{[i+I+N-1,\infty)}) = y_{[i+I+N-1,\infty)}$. Let
    \[
        {\bar x}_k =
            \begin{cases}
                x_k     & \text{if $k \leq i$}   \\
                u_k     & \text{if $i \leq k \leq i+I+N-1 $} \\
                z_k     & \text{if $k \geq i+I+N-1 $}   \\
            \end{cases}
    \]
    Then $\bar x \in X$ since $\alpha$ and $u_{i+I+N-1}$ are synchronizing. Note that $\bar x_{[-n+2N+D,i+I)}$ is a part of long high-stretch and $\bar x_{[i+I,\infty)}$ is a low-stretch (our chosen block $\bar x_i = x_i = a \notin \A(Z)$ guarantees no occurrence of a $Z$-block of length greater than $2N+D$ in $\bar x_{[-n+2N+D,i+I)}$). Therefore rules ii), iv) and i) apply to $\bar x_{[-n+2N+D+I,\infty)}$ and we have $\phi(\bar x) = y$. So $\phi$ is \rct\ with retract $n$. Similarly $\phi$ is \lct\ with retract $n$, which completes the proof when $X$ has the \spe.

    \vspace{0.20cm}
    {\large \textsf{Part} III. \quad} We prove the general case where $X$ is merely \as. Let $\{D_0, D_1, \cdots, D_{p-1} \}$ be the cyclic cover of $X$. Define \[ Z = \bigcup_{i \neq j} D_i \cap D_j. \]
    Then $Z$ is a proper subshift of $X$. By using Lemma \ref{lem:extension}, we may assume that $Z$ is contained in $\widetilde X$.
    Note that $\sigma^p|_{D_0 \cap \widetilde X}$ is (conjugate to) a shift space. Also $\sigma^p|_{D_0}$ has the \spe\ by Lemma \ref{lem:aspe_then_cover_has_specification}. The code $\tilde \phi$ naturally induces a code $\tilde \psi : (D_0 \cap \widetilde X, \sigma^p) \to (Y,\sigma^p)$ by restriction. By applying Parts I and II, we have a $\sigma^p$-commuting \bict\ code $\psi : (D_0, \sigma^p) \to (Y,\sigma^p)$ with a bi-retract such that $\psi|_{D_0 \cap \widetilde X} = \tilde \psi$.
    Define a code $\phi : X \to Y$ as follows: Given $x \in X$, there exists $0 \leq i < p$ with $x \in D_i$. Then we let $\phi(x) = \sigma^{i}\psi(\sigma^{-i}(x))$. If $x \in D_j$ for some $j \neq i$,
    then $x \in Z \subset \widetilde X$ and therefore
    \[ \begin{split}
            \sigma^{j}\psi(\sigma^{-j}(x)) & = \sigma^{j}\tilde \psi(\sigma^{-j}(x)) = \sigma^{j}\tilde \phi(\sigma^{-j}(x)) = \tilde \phi(x) \\
            &= \sigma^{i}\tilde \phi(\sigma^{-i}(x)) =
            \sigma^{i}\tilde \psi(\sigma^{-i}(x)) = \sigma^{i}\psi(\sigma^{-i}(x)).
       \end{split}
    \]
    Thus $\phi$ is well defined. These equalities also show that $\phi|_{\widetilde X} = \tilde \phi$. Note that $\phi$ is continuous. It is easy to see that $\phi$ is indeed $\sigma$-commuting (hence a code) and \bict\ with a bi-retract. By Lemma \ref{lem:open_ull_then_bict_and_converse}, $\phi$ is also an open code with a uniform lifting length.
\end{proof}

\begin{rem}
    Suppose $X$ is assumed to be of finite type in Theorem \ref{thm:exist_bict_code_1}. Then we can take $Z_0 = X$ in Part I. Also Case 2 in Part II can be proved as easily as in Case 1 (since any word is synchronizing) and by using a usual reduction to the mixing case, Part III can be removed. Thus in this case the proof gives an alternative proof of \bict\ version of Theorem \ref{thm:BoyT} in which the definition of $\phi$ is much simpler. This is because by using the extension lemma, we need not consider the marker sets as in \cite{BoyT} and thus $\phi$ on low-high and high-low transitions can be defined more easily.
\end{rem}

\vspace{0.3cm}
\section{Lower entropy factor theorems}\label{sec:main2}

For a synchronized system $X$, fix a synchronized word $w \in \B(X)$. Let $C_n(X)$ be the set of words $v \in \B_n(X)$ such that $wvw \in \B(X)$. Then a synchronized entropy $\hsyn(X)$ is defined by
$$\hsyn(X) = \limsup_{n \to \infty} \frac{1}{n} \log |C_n(X)|.$$
This value is independent of $w$ and $h(X) \geq \hsyn(X)$. In general $h(X) \neq \hsyn(X)$. If $X$ is \as, then by Proposition \ref{prop:aspe_then_SFT_exists} and Theorem 3.2 of \cite{Tho} we have $\hsyn(X) = h(X)$.
By using the cyclic cover of a synchronized system, Thomsen has given a generalization of the lower entropy factor theorem in \cite{Boy84} for the case where $Y$ is merely an irreducible \SFT.

\begin{thm} \cite{Tho} \label{thm:Thomsen_main}
    Let $X$ be a synchronized system and $Y$ an \rSFT\ such that $\hsyn(X) > h(Y)$. Then $Y$ is a factor of $X$ \ifff\ the following hold:
    \begin{enumerate}
        \item[i)] $P(X) \searrow P(Y)$.
        \item[ii)] If $q=\per(Y)$ and $\{D_0, D_1, \cdots, D_{p-1}\}$ is the cyclic cover of $X$, then $q|p$ and
            $$ \biggl( \bigcup_{j=0}^{\frac{p}{q} - 1} D_{i+jq} \biggr) \cap \biggl( \bigcup_{j=0}^{\frac{p}{q} - 1} D_{k+jq} \biggr) = \emptyset$$
            when $i,k \in \{ 0, 1, \ldots, q-1\},\: i \neq k$.
    \end{enumerate}
\end{thm}

Note that when $X$ is of finite type, then the condition (ii) is trivial.
Using this result, we have our main theorem.

\begin{thm}\label{thm:main}
    Let $X$ be an \ass\ and $Y$ an \rSFT\ such that $h(X) > h(Y)$. Then the following are equivalent.
    \begin{enumerate}
        \item $Y$ is a factor of $X$ by a \bict\ code with a bi-retract.
        \item $Y$ is a factor of $X$ by an open code with a uniform lifting length.
        \item $Y$ is a factor of $X$.
    \end{enumerate}

    If $X$ is of finite type, then the above conditions are also equivalent to:
    \begin{enumerate}
        \item[(4)] $P(X) \searrow P(Y)$.
    \end{enumerate}
\end{thm}

\begin{proof}
    (1) $\To$ (2) follows from Lemma \ref{lem:open_ull_then_bict_and_converse}. (2) $\To$ (3) is clear. Now suppose (3). Then we have two conditions in Theorem \ref{thm:Thomsen_main}. Let $\{E_0, E_1, \cdots, E_{q-1}\}$ be the cyclic cover of $Y$.
    Let $C = \bigcup_{j=0}^{\frac{p}{q} - 1} D_{jq}$. Note that $\{C, \sigma (C), \cdots, \sigma^{q-1}(C)\}$ is a partition of $X$ and $( C, \sigma^q)$ is (conjugate to) a shift space. As in the proof of Lemma \ref{lem:aspe_then_cover_has_specification}, it is easy to see that $\sigma^q|_C$ is an \ass.
    Since $q=\per(Y)$, $\sigma^q|_{E_0}$ is a \mSFT.
    By Theorem \ref{thm:exist_bict_code_1}, there exists a ($\sigma^q$-commuting) code $\psi : (C,\sigma^q) \to (E_0,\sigma^q)$ which is \bict\ with a bi-retract.
    As usual, we can define a factor code $\pXtoY$ as follows: For $x \in X$, there exists a unique $0 \leq i < q$ with $x \in \sigma^i(C)$. Define $\phi(x) = \sigma^i \psi(\sigma^{-i}(x))$. It is easy to check that $\phi$ is a ($\sigma$-commuting) code and \bict\ with a bi-retract. Thus (1) holds.

    Finally, when $X$ is of finite type,
    then by using a usual reduction to mixing case (cf. \cite{LM}), the problem can be reduced to Proposition \ref{thm:exist_bict_code_1}. Thus (4) implies (1).
\end{proof}

\begin{rem}
    The equivalence of (1), (2) and (3) in Theorem \ref{thm:main} fails under the hypothesis of Theorem \ref{thm:Thomsen_main}, that is, when $X$ is a synchronized system and $Y$ is an \rSFT\ with $\hsyn(X) > h(Y)$.

    For example, let $X = \X_W$ and $Y = \X_{W_1}$, where $W = \{ ab^k c^k : k \geq 0\}$ and $W_1 = \{ a, abc \}$. Note that $X$ is a synchronized system and $Y$ is an \rSFT. Since $Y$ is a proper subsystem of a mixing sofic shift $\X_{W_2}$, where $W_2 = \{ ab^k c^k : k \leq 2 \}$, it follows that $h(Y) < h(\X_{W_2}) \leq \hsyn(X)$. Since $Y$ has the unique fixed point $a^\infty$, we have $P(X) \searrow P(Y)$ and $Y$ is a factor of $X$.

    We claim that there exists no \rct\ code from $X$ to $Y$. Indeed, if $\phi : X \to Y$ is a code, then let $x = b^\infty$ and $y = a^\infty. \cdots $ where $y_{[0,\infty)}$ is right transitive. Then $\phi(x)$ and $y$ are left asymptotic. But when $z \neq b^\infty$ is left asymptotic to $x$, then $z$ is in the orbit of $b^\infty.c^\infty$, so $\phi(z)$ must be of the form $a^\infty \cdots a^\infty$, since $a^\infty$ is the only fixed point in $Y$. Thus $\phi(z) \neq y$ and $\phi$ cannot be \rct. By Lemma \ref{lem:open_ull_then_bict_and_converse}, $\phi$ cannot be open with a uniform lifting length.

    As we have shown in Example \ref{ex:examples_section2} (1), there is an open code from $X$ to $Y$. At this time, it is an open question whether there exists an open factor code from a synchronized system $X$ to an \rSFT\ $Y$ if there is a factor code from $X$ to $Y$.
\end{rem}

From Lemma \ref{prop:factor_of_SFT_is_SFT} and Theorem \ref{thm:main}, the following result is immediate.

\begin{cor} \label{thm:exist_open_code}
    Let $X$ be an \rSFT\ and $Y$ a shift space. Then $Y$ is a lower entropy open factor of $X$ \ifff\ $h(X) > h(Y)$, $P(X) \searrow P(Y)$ and $Y$ is an \rSFT.
\end{cor}

Let $X$ and $Y$ be synchronized systems with their cyclic covers  $\{D_0, D_1, \cdots, D_{p-1}\}$ and $\{E_0, E_1, \cdots, E_{q-1}\}$, respectively. If $\phi : X \to Y$ is a code, then there exists an integer $ 0 \leq j < q$ with $\phi(D_0) \subset E_j$
since $\sigma^{pq}$ acts transitively on $D_i$.
So for a code $\tilde \phi : \widetilde X \to Y$ from a proper subshift $\widetilde X$ of $X$ to be extended to a code on $X$, a necessary condition is that there exists an integer $ 0 \leq j < q$ with
$$\tilde \phi(\widetilde X \cap D_0) \subset E_j.$$
When this holds, we say that $\tilde \phi$ satisfies the \emph{cyclic condition for $X$}. Note that if $Y$ is mixing, then any code from a proper subshift of $X$ to $Y$ satisfies the cyclic condition for $X$.

The following extension theorem says that if the conditions in Theorem \ref{thm:main} hold, then the above cyclic condition is the only obstruction to extend a code defined on a proper subshift of $X$ to a code from $X$ to $Y$. It is this obstruction which is responsible for the failure of Theorem \ref{thm:BoyT} and \ref{thm:exist_bict_code_1} when the target is not mixing.

\begin{thm}\label{thm:main2}
    Let $X$ be an \ass\ and $Y$ an \rSFT\ such that $h(X) > h(Y)$ and $Y$ is a factor of $X$. For a code $\tilde \phi : \widetilde X \to Y$ defined on a proper subshift $\widetilde X$ of $X$, \tFAE.
    \begin{enumerate}
        \item $\tilde \phi$ satisfies the cyclic condition for $X$.
        \item There is a code $\psi : X \to Y$ with $\psi|_{\widetilde X} = \tilde \phi$.
        \item There is a factor code $\pXtoY$ with $\phi|_{\widetilde X} = \tilde \phi$ which is open with a uniform lifting length and which is \bict\ with a bi-retract.
    \end{enumerate}
\end{thm}
\begin{proof}
    It is clear that $(3) \To (2) \To (1)$. So assume (1) and let $\{E_0, E_1, \cdots, E_{q-1}\}$ be the cyclic cover of $Y$. Since $Y$ is a factor of $X$, we have the two conditions in Theorem \ref{thm:Thomsen_main}. Let $C = \bigcup_{j=0}^{\frac{p}{q} - 1} D_{jq}$. Then $\sigma^q|_C$ is an \ass. Let $\tilde \psi$ be the restriction of $\tilde \phi$ to $C$.
    Since $\tilde \phi$ satisfies the cyclic condition for $X$, there exists an element $E_j$ in the cyclic cover of $Y$ with $\tilde \phi(D_0 \cap \widetilde X) \subset E_j$. Then we have $\tilde \phi(C \cap \widetilde X) \subset E_j$. Note that $\sigma^q|_{E_j}$ is a mixing \SFT.

    By Theorem \ref{thm:exist_bict_code_1}, we have a ($\sigma^q$-commuting) code $\psi : (C,\sigma^q) \to (E_j,\sigma^q)$ which is an extension of $\tilde \psi$ and \bict\ with a bi-retract. Remainder of the proof goes as in that in Theorem \ref{thm:main}.
\end{proof}

We remark that in \cite{BarD}, Barth and Dykstra considered related conditions in the case of reducible \SFTs\ using \emph{phase matrices}.

\vspace{0.1cm}
Our last example shows that even when both shift spaces are of finite type, the cyclic condition is very crucial to extend a code defined on a proper subshift of the domain. In contrast to the existence theorems, extension theorems do not directly follow from the mixing case. 

\begin{exam} \label{ex:extensionfail}
    Let $X = \X_A$ with $A = \left( \begin{smallmatrix} 0 & 2 \\ 2 & 0 \end{smallmatrix} \right)$ and $Y = \{(ab)^\infty, (ba)^\infty \}$. The edges for $X$ are given as in the following figure. Then $P(X) \searrow P(Y)$ and $Y$ is a factor of $X$.
    Let $\widetilde X = \{ (13)^\infty, (31)^\infty, (24)^\infty, (42)^\infty \}$. We define $\tilde \phi : \widetilde X \to Y$ by
    \[ \begin{split}
            \tilde \phi ( (13)^\infty) = (ab)^\infty, \qquad & \tilde \phi ( (31)^\infty) = (ba)^\infty, \\
\tilde \phi ( (24)^\infty) = (ba)^\infty, \qquad & \tilde \phi ( (42)^\infty) = (ab)^\infty.            \end{split}
    \]

    \vspace{-0.5cm}
    \begin{Figure}[h]
        \label{fig:extensionfail}
        \includegraphics[height=2.5cm]{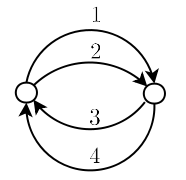}
    \end{Figure}
    \vspace{-0.3cm}

    Suppose $\pXtoY$ is a code with $\phi|_{\widetilde X} = \tilde \phi$. Write $X = D_0 \dot \cup D_1$ where $D_0 = \{ x \in X : x_0 = 1 \text{ or } 2 \}$ and $D_1 = \sigma (D_0)$. Then $(13)^\infty, (24)^\infty \in D_0$. Since $\sigma^2|_{D_0}$ is irreducible, $\sigma^2|_{\phi(D_0)}$ is also irreducible, hence $\phi(D_0)$ must be a single point. This is a contradiction. Thus $\tilde \phi$ cannot be extended to a code on $X$.
\end{exam}

\textit{Acknowledgment.} This paper was written as part of the author's doctorial thesis, under the guidance of Prof. Sujin Shin. I would like to thank for all her suggestions and good advice. Thanks also to Jisang Yoo for valuable comments. Also I am grateful to Kim Bojeong Basic Science Foundation for generous help. Partial support was provided by the second stage of the Brain Korea 21 Project, The Development Project of Human Resources in Mathematics, KAIST in 2008.

\bibliographystyle{amsplain}

\end{document}

%% file: abbrev_1.1.tex
\newcommand{\setN}{\mathbb{N}}

\newcommand{\setZ}{\mathbb{Z}}

\newcommand{\Zplus}{\mathbb{Z^+}}

\newcommand{\A}{\mathcal{A}}    
\newcommand{\B}{\mathcal{B}}    
\newcommand{\F}{\mathcal{F}}    

\newcommand{\X}{\mathsf{X}}

\newcommand{\ifff}{if and only if}

\newcommand{\cto}{constant-to-one}
\newcommand{\fto}{finite-to-one}
\newcommand{\ito}{infinite-to-one}

\newcommand{\bir}{bi-resolving}
\newcommand{\rc}{right closing}
\newcommand{\lc}{left closing}
\newcommand{\bic}{bi-closing}
\newcommand{\rct}{right continuing}
\newcommand{\lct}{left continuing}
\newcommand{\bict}{bi-continuing}
\newcommand{\SFT}{shift of finite type}
\newcommand{\SFTs}{shifts of finite type}
\newcommand{\rSFT}{irreducible shift of finite type}
\newcommand{\rSFTs}{irreducible shifts of finite type}
\newcommand{\mSFT}{mixing shift of finite type}

\newcommand{\Sofic}{sofic shift}
\newcommand{\rSofic}{irreducible sofic shift}

\newcommand{\tFAE}{the following are equivalent}

\newcommand{\set}[1]{\left\{#1\right\}}
\newcommand{\per}{\mathrm{per}}
\newcommand{\To}{\Rightarrow}

\newcommand{\pXtoY}{\phi : X \to Y}

\newcommand{\preimageAbs}[2] {|{#1}^{-1}(#2)|}

\newcommand{\onto}{\xymatrix{\ar@{>>}[r]&}}
\newcommand{\da}[4]{\xymatrix{#1 \ar@<.5ex>[r]^{#2} \ar@<-.5ex>[r]_{#3} & #4}}